\documentclass[12pt]{amsart}

\allowdisplaybreaks[1]
\usepackage{enumerate}
\usepackage{graphicx}
\usepackage{amsmath}
\usepackage{amssymb}
\usepackage{amsfonts}
\usepackage[dvipsnames]{xcolor}
\usepackage{hyperref}
\allowdisplaybreaks

 \newtheorem{theorem}{Theorem}[section]

 \newtheorem{proposition}[theorem]{Proposition}

\theoremstyle{definition}

\theoremstyle{remark}

\newtheorem{fact*}{Fact}

\DeclareMathOperator{\IM}{Im}

\newcommand{\C}{\mathbb{C}}

\newcommand{\PPT}[1]{\operatorname{PPT}#1}

\newcommand{\inv}{^{-1}}

\newcommand{\til}{\raise.17ex\hbox{$\scriptstyle\mathtt{\sim}$}}

\newcommand\beq{\begin{equation}}

\newcommand\eeq{\end{equation}}

\newcommand{\bbm}{\left[ \begin{smallmatrix}}
\newcommand{\ebm}{\end{smallmatrix} \right]}
\newcommand{\bpm}{\left( \begin{smallmatrix}}
\newcommand{\epm}{\end{smallmatrix} \right)}
\numberwithin{equation}{section}

\newlength{\Mheight}
\newlength{\cwidth}

\newcommand{\dfn}[1]{{\bf #1}\index{#1}}

\newcommand{\Mn}{M_n(\C)}

\title[Principal pivot monotone]{Monotonicity of the principal pivot transform}
\author[J. E. Pascoe]{
J. E. Pascoe$^\dagger$
}
\address{Department of Mathematics\\
1400 Stadium Rd\\
  University of Florida\\
 Gainesville, FL 32611}
\email[J. E. Pascoe]{pascoej@ufl.edu}
\thanks{$\dagger$ Partially supported by National Science Foundation Mathematical
Science Postdoctoral Research Fellowship  
DMS 1606260 and DMS Analysis Grant 1953963}

\author[R. Tully-Doyle]{
Ryan Tully-Doyle$^\ddagger$
}
\address{Mathematics Department \\
1 Grand Ave \\
Cal Poly, SLO\\
San Luis Obispo, CA 93407}
\email[R. Tully-Doyle]{rtullydo@calpoly.edu}
\thanks{$\ddagger$ Partially supported by National Science Foundation DMS Analysis Grant 2055098 and University of New Haven SRG and Cal Poly startup}

\date{\today}

\subjclass[2020]{15A09, 47A56}

\begin{document}

\begin{abstract}
We prove that the principal pivot transform (also known as the partial inverse, sweep operator, or exchange operator in various contexts) maps matrices with positive imaginary part to matrices with positive imaginary part. We show that the principal pivot transform is matrix monotone by establishing Hermitian square representations for the imaginary part and the derivative.
\end{abstract}

\maketitle

\section{Introduction} Suppose that $A \in \Mn$ is partitioned into the block matrix
\[
A = \bbm A_{11} & A_{12} \\ A_{21} & A_{22} \ebm
\]
and that the block $A_{11}$ is an invertible matrix. The \dfn{principal pivot transform} of $A$ is the matrix $B$ given by 
\[
\PPT(A) = \bbm -A_{11}\inv & A_{11}\inv A_{12} \\ A_{21} A_{11}\inv & A_{22} - A_{21} A_{11}\inv A_{12} \ebm.
\]
Applied to a linear equation $A x = y$, the principal pivot transform has the effect of switching the place of initial and target data in the first slot and then negating the first slot. That is, after partitioning $x, y$ relative to $A$,
\beq \label{pptformula}
A \bbm x_1 \\ x_2 \ebm = \bbm y_1 \\ y_2 \ebm \text{ if and only if } \PPT(A) \bbm y_1 \\ x_2 \ebm = \bbm -x_1 \\ y_2 \ebm.
\eeq
We note that some authors use a slightly different version of the principal pivot transform where the top block row is negated. The principal pivot transform has been thoroughly studied in various guises, for example, as the sweep operator \cite{lange, tao} and the partial inverse \cite{wer1, wer2}. The PPT is used in many applications (see \cite{pol, tsat} for a wide list of references including control theory, numerical analysis, and linear regression).

Equation \ref{pptformula} implies that
	\[\PPT(\PPT(\PPT(\PPT(A)))) = A.\]
Namely, the principal pivot transform algebraically has order $4$ and is injective on its domain.
Define the \dfn{imaginary part} of a matrix $A$ to be
	\[\IM A = \frac{A-A^*}{2i},\]
where $A^*$ denotes the adjoint or conjugate transpose of $A.$
Let $A, B$ be self-adjoint matrices. We say that $A\preceq B$ if $B-A$ is positive semidefinite. This ordering on matrices, the so-called L\"owner ordering, is related to majorization of eigenvalues and therefore other applications \cite{ando}.

We show that the principal pivot transform is an automorphism of matrices with positive imaginary part.
\begin{theorem}\label{mainresult}
	The following are true:
	\begin{enumerate}
		\item Suppose $A$ is a block $2$ by $2$ matrix. If $\IM A$ is positive semidefinite, then $\IM \PPT(A)$ is positive semidefinite.
		\item Suppose $A, B$ are block $2$ by $2$ matrices, such that for every $t\in[0,1]$ the matrix $(1-t)A_{11} + tB_{11}$ is invertible. If $A \preceq B,$ then $\PPT(A) \preceq \PPT(B).$
	\end{enumerate}
\end{theorem}
Part (1) follows from Proposition \ref{improp}. Part (2) comes integrating a formula for the derivative of the principal pivot transform from Proposition \ref{deprop}.
In the case where each of the blocks are square matrices, condition (1) implies condition (2) in Theorem \ref{mainresult} by the noncommutative L\"owner theorem \cite{PTDPick, PTDRR, palfia, pascoeopsys}. 
We note that as the principal pivot transform is an automorphism of the block $2$ by $2$ matrix ``upper half plane", under conjugation by a suitable Cayley transform it is conjugate to an automorphism of
block $2$ by $2$ matrices which was studied in \cite{HKMS09}. 
Matrix monotonicity is important in various contexts including MIMO systems (see, e.g. \cite{MIMO, GATORHULLS}).

\section{Hermitian square representations}
The proof relies on the basic observation that 
if we can write $Y = Z^*XZ$ and $X$ is positive semidefinite, then $Y$ is also positive semidefinite.

\begin{proposition}\label{improp}
	Let $A$ be a square block $2$ by $2$ matrix such that $A_{11}$ is invertible.
	\[\IM \PPT(A) =  \bbm A_{11}^{-1} & -A_{11}^{-1}A_{12} \\  0 & 1 \ebm^* \IM A \bbm A_{11}^{-1} & -A_{11}^{-1}A_{12} \\  0 & 1 \ebm.\]
\end{proposition}
\begin{proof}
	Note
	\[ 2i\IM A = \bbm A_{11}-A_{11}^* & A_{12}-A_{21}^* \\ A_{21}-A_{12}^* & A_{22}-A_{22}^* \ebm\]
	Now,
	\[ 2i\IM A  \bbm A_{11}^{-1} & -A_{11}^{-1}A_{12} \\  0 & 1 \ebm =
	\bbm 1-A_{11}^*A_{11}^{-1}  & A_{11}^* A_{11}^{-1}A_{12}-A_{21}^*  \\ A_{21}A_{11}^{-1}-A_{12}^*A_{11}^{-1} & -A_{21} A_{11}^{-1}A_{12}+A_{12}^* A_{11}^{-1}A_{12}  + A_{22}-A_{22}^* \ebm.\]
Now considering,
\[ \bbm A_{11}^{-1} & -A_{11}^{-1}A_{12} \\  0 & 1 \ebm^*  = \bbm (A_{11}^*)^{-1} & 0 \\   -A_{12}^*(A_{11}^*)^{-1} & 1 \ebm,\]
and calculating,
\[\bbm (A_{11}^*)^{-1} & 0 \\   -A_{12}^*(A_{11}^*)^{-1} & 1 \ebm\bbm 1-A_{11}^*A_{11}^{-1}  & A_{11}^* A_{11}^{-1}A_{12}-A_{21}^*  \\ A_{21}A_{11}^{-1}-A_{12}^*A_{11}^{-1} & -A_{21} A_{11}^{-1}A_{12}+A_{12}^* A_{11}^{-1}A_{12}  + A_{22}-A_{22}^* \ebm,
\]
we get
\[
	\bbm
		 -A_{11}^{-1}+  (A_{11}^*)^{-1} & A_{11}^{-1}A_{12} - (A_{11}^*)^{-1}A_{21}^*\\
		 A_{21}A_{11}^{-1} -A_{12}^* (A_{11}^*)^{-1} & 
		A_{22} - A_{21} A_{11}\inv A_{12} - A_{22}^{*} + A_{12}^* (A_{11}^*)\inv A_{21}^*
	\ebm,
\]
which is exactly $2i \IM \PPT(A).$ 
\end{proof}

\begin{proposition}\label{deprop}
	Let $A, H$ be like-sized self-adjoint square block $2$ by $2$ matrices such that $A_{11}$ is invertible.
	\[D\PPT(A)[H] =  \bbm A_{11}^{-1} & -A_{11}^{-1}A_{12} \\ 0 & 1 \ebm^* H \bbm A_{11}^{-1} & -A_{11}^{-1}A_{12} \\ 0 & 1 \ebm\]
	where $D\PPT(A)[H] = \frac{d}{dt} PPT(A+tH) |_{t=0}.$
\end{proposition}
\begin{proof}
	Evaluate the formula from Proposition \ref{improp} at $A^{(t)}=A+itH$ to obtain
		\[\IM \PPT(A^{(t)}) = 
	\bbm (A^{(t)}_{11})^{-1} & -(A^{(t)}_{11})^{-1}A^{(t)}_{12} \\  0 & 1 \ebm^* \IM A^{(t)} \bbm (A^{(t)}_{11})^{-1} & -(A^{(t)}_{11})^{-1}A^{(t)}_{12} \\  0 & 1 \ebm.\]
	Dividing by $t$, we get
	\[\frac{ \IM \PPT(A^{(t)})}{t} = 
	\bbm (A^{(t)}_{11})^{-1} & -(A^{(t)}_{11})^{-1}A^{(t)}_{12} \\  0 & 1 \ebm^*  \frac{ \IM A^{(t)}}{t} \bbm (A^{(t)}_{11})^{-1} & -(A^{(t)}_{11})^{-1}A^{(t)}_{12} \\  0 & 1 \ebm.\]
	Note that
		\[ \frac{ \IM A^{(t)}}{t} = H, 
		\]
		and 
		\[\frac{ \IM \PPT(A^{(t)})}{t} = \frac{\PPT(A+itH)-\PPT(A-itH)}{2it}.\]
	By taking the limit as $t$ goes to $0$, we get 
		\[D\PPT(A)[H] =  \bbm A_{11}^{-1} & -A_{11}^{-1}A_{12} \\ 0 & 1 \ebm^* H \bbm A_{11}^{-1} & -A_{11}^{-1}A_{12} \\ 0 & 1 \ebm.\]
\end{proof}

\bibliography{references}
\bibliographystyle{plain}


\end{document}